\documentclass[a4paper,12pt]{article}

\usepackage{color}
\usepackage[pdftex]{graphicx}
\usepackage{amsmath,amssymb,amsthm}
\usepackage{wrapfig}
\usepackage{ifpdf}

\newtheorem{theorem}{Theorem}[section]
\newtheorem{cor}{Corollary}[section]
\newtheorem{lemma}{Lemma}[section]
\newtheorem{prop}{Proposition}[section]

\newtheorem{df}{Definition}[section]

\newcommand{\conv}{\mathop{\rm conv}\nolimits}

\newcommand{\cov}{\mathop{\rm cov}\nolimits}

\newcommand{\ep}{\mathop{\rm EP}\nolimits}

\newcommand{\vrt}{\mathop{\rm Vert}\nolimits}
\newcommand{\st}{\mathop{\rm St}\nolimits}
\newcommand{\bd}{\mathop{\rm Bd}\nolimits}

\DeclareMathOperator{\dg2}{deg_2}

\title {Homotopy invariants of covers and KKM type lemmas}

\author {Oleg R. Musin\thanks{The research was carried out at the IITP RAS at the expense of the Russian Foundation for Sciences (project № 14-50-00150).}}

\begin{document}

	\ifpdf \DeclareGraphicsExtensions{.pdf, .jpg, .tif, .mps} \else
	\DeclareGraphicsExtensions{.eps, .jpg, .mps} \fi	
	
\date{}
\maketitle

\begin{abstract} 
Given any (open or closed) cover of a space $T$ we associate certain homotopy classes of maps from $T$ to $n$--spheres. 
These homotopy invariants can then be considered as obstructions for extending covers of a subspace $A \subset X$ to a cover of all of $X$. 
We are using these obstructions to obtain generalizations of the classic KKM (Knaster--Kuratowski--Mazurkiewicz) and Sperner lemmas. 
In particular, we show that in the case when $A$ is a $k$--sphere and $X$ is a $(k+1)$--disk there exist KKM type lemmas for covers by $n+2$ sets if and only if the homotopy group $\pi_{k}({\Bbb S}^{n})\ne0$. 
\end{abstract}

\medskip


\medskip 

\medskip 

Throughout this paper we will consider only normal topological spaces, all simplicial complexes will be finite, all manifolds with be both compact and PL,  ${\Bbb S}^{n}$ will denote the $n$--dimensional unit sphere, and ${\Bbb B}^{n}$ will denote the $n$--dimensional unit disk. 
We shall denote the set of homotopy classes of continuous maps from X to Y by $[X,Y]$.

\section{Homotopy invariants of covers}

First we consider labelings (colorings) of simplicial complexes. 
Denote by  $\vrt(K)$ the vertex set of a simplicial complex $K$. 
(It is also referred to as the  0--skeleton, $K^0$.) 
Let $$ L:\vrt(K)\to\{0,1,\ldots,m\} $$ be a labeling of vertices of $K$. 

Denote by $\Delta^m$ an $m$--dimensional simplex with vertices $v_0,\ldots,v_m$. 
Let 
$$f_L(u):=v_\ell, \; \mbox{ where } \; u\in\vrt(K), \, \text{ and } \ell=L(u). $$
Since $f_L$ is defined for all of the vertices of $K$, it induces a simplicial mapping $f_L:|K|\to |\Delta^m|$. This map is unique up to homeomorphism. 

Note that if any simplex in $K$ has at most $m$ distinct labels, then $f_L$ is a map from $|K|$ to $\partial|\Delta^m|\cong {\Bbb S}^{m-1}$. 

\begin{df} 
For a simplicial complex $K$ and a labeling $L:\vrt(K)\to\{0,1,\ldots,m\}$ such that $K$ has no simplices with $m+1$ distinct labels we denote by $[L]$ the homotopy class $[f_L]$ in $[|K|,{\Bbb S}^{m-1}]$. 
\end{df}

\noindent{\bf Example 1.1.} Let $K$ be a triangulation of ${\Bbb S}^k$ and   $L:\vrt(K)\to\{0,1,\ldots,m\}$ be a labeling such that $K$ has no simplices with $m+1$ distinct labels. Then $[L]\in\pi_k({\Bbb S}^{m-1})$. 

In the case $k=m-1$ we have $\pi_k({\Bbb S}^{m-1})={\Bbb Z}$ and 
$$[L]=\deg(f_L)\in{\Bbb Z}.$$
(Here by $\deg(f)$ we denote the degree of a continuous map $f$ from  ${\Bbb S}^n$  to itself.)

For instance, let $L$ be a {\it Sperner labeling} of a triangulation $K$ of $\partial\Delta^m=u_0u_1\ldots u_m$. The rules of this labeling are:\\
(i) The vertices of $\Delta^m$ are colored with different colors, i. e. $L(u_i)=i$ for $0\le i \le m$.\\
(ii)   Vertices of $K$ located on any $n$-dimensional subface of the large simplex
$u_{i_0}u_{i_1}\ldots u_{i_{n}}$ are colored only with the colors $i_0,i_1,\ldots,i_{n}.$ 
\\ Then $[L]=\deg(f_L)=1$ in $[{\Bbb S}^{m-1},{\Bbb S}^{m-1}]={\Bbb Z}$. 

\medskip

\noindent{\bf Example 1.2.}  Madahar and Sarkaria \cite{MS} considered a simplicial  map
 $\tau:{\Bbb S}_{12}^3\to{\Bbb S}_4^2$ from a 12-vertex 3-sphere ${\Bbb S}_{12}^3$ onto the 4-vertex 2--sphere ${\Bbb S}_4^2$ (tetrahedron) with vertices $v_0v_1v_2v_3$. 
 Actually, $\tau$ is a minimal vertices simplicial map of the Hopf invariant one. 

For $u\in\vrt({\Bbb S}_{12}^3)$, let
$L_\tau(u):=i,  \mbox{ where }  \tau(u)=v_i.$ 
Then $[L_\tau]=1\in\pi_3({\Bbb S}^2)={\Bbb Z}$. 

\medskip

Let $K$ be a simplicial complex. 
Denote by $\st(u)$ the open star of a vertex $u\in\vrt(K)$. 
In other words, $\st(u)$ is $|S|\setminus |B|$, where $S$  is the set of all simplices in $K$ that contain $u$, and $B$ is the set of all simplices in $S$ that does contain $u$.  

Let $$ L:\vrt(K)\to\{0,1,\ldots,m\} $$ be a labeling of vertices of $K$. 
There is a natural open cover of $|K|$ $$ \mathcal U_L(K)=\{U_0(K),\ldots,U_m(K)\},$$  where 
$$U_\ell(K):=\bigcup\limits_{u\in W_\ell}{\st(u)}, \; \; \text{ and } W_\ell:=\{u\in\vrt(K): L(u)=\ell\}. $$

So with any labeling $L$ we associate a cover $\mathcal U_L(K).$ 
Now we extend Definition 1.1 to covers. 

Let $\mathcal U=\{U_0,\ldots,U_m\}$ be an open finite cover of a  space $T$. 
If $N(\mathcal U)$ is its nerve, then there is a one-to-one correspondence between canonical maps $c:T\to |N(\mathcal U)|$ and partitions of unity $\Phi$ subordinate to $\mathcal U$. 
Moreover, any two canonical maps $T\to |N(\mathcal U)|$ are homotopic. 

Since the nerve $N(\mathcal U)$ is a subcomplex of the simplex $\Delta^m$, we have an embedding $\alpha:|N(\mathcal U)|\to |\Delta^m|$. 
In the case when the intersection of all of the $U_i$ is empty, i. e.  $N(\mathcal U)$ does not contain an $m$--cell, we have
$$\alpha:|N(\mathcal U)|\to \partial|\Delta^m|\cong {\Bbb S}^{m-1}.$$ 
If 
$$\rho_{\mathcal U,c}:=\alpha\circ c,$$ 
then a homotopy class  $[\rho_{\mathcal U,c}]$ in $[T,{\Bbb S}^{m-1}]$ does not depend on the canonical map  $c:T\to |N(\mathcal U)|$.

\begin{df}  
Let $\mathcal U=\{U_0,\ldots,U_m\}$ be an open finite cover of a  space $T$ such that the intersection of all of the $U_i$ is empty. 
Denote by 
$[\mathcal U]$ the homotopy class  $[\rho_{\mathcal U,c}]$ in $[T,{\Bbb S}^{m-1}]$. 
\end{df}

\noindent{\bf Remark.} It is clear that 
$$[\mathcal U_L(K)] = [L] \mbox{ in } [|K|,{\Bbb S}^{m-1}]. $$

\begin{theorem} 
Let $T$ be a space and $h$ be a homotopy class in $[T,{\Bbb S}^{m-1}]$. 
Then there is an open cover $\mathcal U=\{U_0,\ldots,U_m\}$ such that $[\mathcal U]=h$. 
	
If $T$ is a simplicial complex, then there is a triangulation $K$ of $T$ and a labeling $L:\vrt(K)\to\{0,1,\ldots,m\}$ with  $[\mathcal U_L(K)]=h$.  
\end{theorem}

\begin{proof} 
Let $\Lambda:\vrt(\Delta^m)\to\{0,1,\ldots,m\}$ be a labeling of $\Delta^m$ with vertices $v_0,\ldots,v_m$ such that $\Lambda(v_\ell)=\ell$ for all $\ell$. 
Then we have a cover $\mathcal U_\Lambda(\Delta^m)$. 
	
Let $f:T\to{\Bbb S}^{m-1}$ be a continuous map with $[f]=h$ and 
$$
U_\ell:=f^{-1}(U_\ell(\Delta^m)), \; \ell=0,\ldots,m.
$$ 
It is easy to see that $[\mathcal U]=h$. 

If $T$ is a simplicial complex, then by the simplicial approximation theorem there is a simplicial subdivision (triangulation) $K$ and a simplicial map $g:K\to\Delta^m$ such that $g$ is homotopic to $f$. 
For all $v\in\vrt(K)$, let
$$L(v):=\Lambda(g(v)).$$
Then $[\mathcal U_L(K)]=h$.  
\end{proof}

Let us define the class $[\mathcal U]$ more explicitly. 
Let $\Phi=\{\varphi_0,\ldots,\varphi_m\}$ be a partition of unity subordinate to $\mathcal U$ and for all $x\in T$
$$
\rho_{\mathcal U,\Phi}(x):=\sum\limits_{i=0}^m{\varphi_i(x)v_i}, 
$$
where $v_0,\ldots,v_m$ are vertices of an $m$--dimensional simplex $V$ considered as vectors in ${\Bbb R}^{m}$. 
Then $\rho_{\mathcal U,\Phi}$ is a continuous map from $T$ to $\partial V={\Bbb S}^{m-1}$. 
It is clear that 
 $$[\rho_{\mathcal U,\Phi}]=[\mathcal U] \mbox{ in } [T,{\Bbb S}^{m-1}]. $$ 

\medskip

Now we extend this definition. Let $V:=\{v_0,\ldots, v_m\}$ be any set of points (vectors) in ${\Bbb R}^{n+1}$. 
As above 
$$\rho_{\mathcal U,\Phi,V}(x):=\sum\limits_{i=0}^m{\varphi_i(x)v_i}. $$

Suppose a point $p\in {\Bbb R}^{n+1}$ lies outside of the image $\rho_{\mathcal U,\Phi,V}(T)$. 
For all $x\in T$, let
$$f_{\mathcal U,\Phi,V,p}(x):=\frac{\rho_{\mathcal U,\Phi,V}(x)-p}{||\rho_{\mathcal U,\Phi,V}(x)-p||}. $$
Then $f_{\mathcal U,\Phi,V,p}$ is a continuous map from $T$ to ${\Bbb S}^n$. 

\begin{lemma} 
For given \, $\mathcal U, V$ and $p$ any two  partitions of unity subordinate to $\mathcal U$ define the same homotopy class $[f_{\mathcal U,V,p}]$ in $[T,{\Bbb S}^n]$.
\end{lemma}

\begin{proof} 
A linear homotopy $\Theta(t)=(1-t)\Phi+t\Psi$ of two partitions of unity $\Phi$ and $\Psi$ induces a homotopy between the maps $f_{\mathcal U,\Phi,V,p}$ and $f_{\mathcal U,\Psi,V,p}$. 
\end{proof}

\begin{lemma} 
For  any two  partitions of unity $\Phi$ and $\Psi$ subordinate to $\mathcal U$, the image $\rho_{\mathcal U,\Phi,V}(T)$ coincides with the image  $\rho_{\mathcal U,\Psi,V}(T)$ in ${\Bbb R}^{n+1}$. 
\end{lemma}

\begin{proof} 
Consider the nerve $N(\mathcal U)$ with vertices $U_i$. 
If we set  $g(U_i):=v_i$, then we have a piece-wise linear map $g:|N(\mathcal U)|\to H$, where $H:=\conv(V)$ is the convex hull of $V$ in ${\Bbb R}^{n+1}$. 
Then for any partition of unity $\Phi$ we have $\rho_{\mathcal U,\phi,V}:=g\circ c$, where $c:T\to |N(\mathcal U)|$ is the canonical map corresponding to $\Phi$. Thus, $\rho_{\mathcal U,\Phi,V}(T)=g(|N(\mathcal U)|)$ and  does not depend on $\Phi$. 
\end{proof}

\noindent {\bf Notation:} $P_{\mathcal U,V}(T):={\Bbb R}^{n+1}\setminus \rho_{\mathcal U,\Phi,V}(T)$.

\medskip

Note that the map $f_{\mathcal U,\Phi,V,p}:T\to{\Bbb S}^n$ is well defined only if $p\in P_{\mathcal U,V}(T)$.  

\begin{lemma} 
Let points $p$ and $q$ lie in the same connected component $Q$ of $P_{\mathcal U,V}(T)$. 
Then $[f_{\mathcal U,V,p}]=[f_{\mathcal U,V,q}]$ in $[T,{\Bbb S}^n]$.
\end{lemma}

\begin{proof} 
Let $s(t)$ be a path in $Q$ connecting the points $p$ and $q$. 
Then $s$ induces a homotopy between the maps $f_{\mathcal U,\Phi,V,p}$ and $f_{\mathcal U,\Phi,V,q}$.
\end{proof}

\begin{df} 
For a  cover $\mathcal U=\{U_1,\ldots,U_m\}$ of a  space $T$, a set $V$ of $m$ points in ${\Bbb R}^{n+1}$, and $p\in P_{\mathcal U,V}(T)$ denote the homotopy class $[f_{\mathcal U,V,p}]$ in $[T,{\Bbb S}^n]$ by $h(\mathcal U,V,p)$. 
	
For a labeling $L:\vrt(K)\to\{0,1,\ldots,m\}$ of a simplicial complex $K$ we denote by $h(K,L,V,p)$ the homotopy class $h(\mathcal U_L(K),V,p)$ in $[|K|,{\Bbb S}^{n}]$. 
\end{df}

\noindent{\bf Example 1.3.} Let $K$ be a heptagon with seven consecutive vertices labelled as $0,1,2,3,2,1,3$.
Let $V=\{v_0,v_1,v_2,v_3\}$ be the set of vertices of a planar square. 
Then $h(K,L,V,p)=1$ if $p$ lies in the triangle $v_0v_1v_3$ and $h(K,L,V,p)=0$ otherwise. 

\medskip

Now we consider homotopy classes of covers of closed sets. 
Let $\mathcal C=\{C_0,\ldots,C_m\}$ be a closed cover of a space $T$. 
Let $\mathcal U=\{U_0,\ldots,U_m\}$ be an open cover of $T$ such that $U_i$ contains $C_i$ for all $i$. 
We say that $\mathcal U$ {\it contains} $\mathcal C$. 

We may assume that the nerves $N(\mathcal U)$ and $N(\mathcal C)$ are isomorphic. 
Otherwise, if there is a subset of indices $J\subset \{0,\ldots,m\}$ such that the intersection of those $U_i$ whose subindices are in $J$ is non-empty and the intersection of those $C_i$ whose subindices are in $J$ is empty we consider an open cover $\mathcal U'$ with 
$$
U_i':=U_i\setminus K_J, \; \; \text{ and } K_j:=\bigcap\limits_{j\in J}{\bar U_j}.
$$ 
Since $C_i\cap K_J=\emptyset$, we have that $U'_i$ contains $C_i$. 

Suppose two open covers $\mathcal U^1$ and $\mathcal U^2$ which both contain $\mathcal C$ and that$N(\mathcal U^1)$, $N(\mathcal U^2)$ and $N(\mathcal C)$ are isomorphic. 
Then the cover $\mathcal U^3$ such that each $U^3_i:=U^1_i\cap U^2_i$ also contains $\mathcal C$ moreover $N(\mathcal U^3)$ is isomorphic to $N(\mathcal C)$. 
Since  both $\mathcal U^1$ and $\mathcal U^2$ contain $\mathcal U^3$, we have the equalities $h(\mathcal U^1,V,p)=h(\mathcal U^3,V,p)= h(\mathcal U^2,V,p)$. 

This observation proves the following statement. 
\begin{lemma} 
\label{prop3} 
Let $\mathcal C$ be a closed cover of a normal space $T$. 
Then there exist an open cover, $\mathcal U$, of $T$ which contains $\mathcal C$ and such that the nerves $N(\mathcal U)$ and $N(\mathcal C)$ are isomorphic.
	
If open covers $\mathcal U^1$ and $\mathcal U^2$ both contain $\mathcal C$ and the nerves $N(\mathcal U^1)$, $N(\mathcal U^2)$ and $N(\mathcal C)$ are isomorphic, then $h(\mathcal U^1,V,p)=h(\mathcal U^2,V,p)$. 
\end{lemma}

This lemma shows that the homotopy class $h(\mathcal C,V,p):=h(\mathcal U,V,p)$ in $[T,{\Bbb S}^n]$ where $N(\mathcal U)=N(\mathcal C)$ and such that  $\mathcal U$ contains $\mathcal C$ is well defined. 

%


\section{Extension of covers}

In this section we consider extensions of covers of a subspace $A$ to a space $X$. 

We call a family of sets a {\it cover} of a space if it is either an open or closed cover.

\begin{df}   
Let $A$ be a subspace of a space $X$. 
Let $\mathcal S=\{S_0,\ldots,S_m\}$ be a cover of $A$ and $\mathcal F=\{F_0,\ldots,F_m\}$ be a cover of $X$. 
We assume that $\mathcal F$ is open if $\mathcal S$ is open and closed if $\mathcal S$ is closed. 	
We say that $\mathcal F$  is an extension of  $\mathcal S$ if  
$$S_i=F_i\cap A \, \mbox{ for all }\, i.$$ 
\end{df}

We start from the classic case:  $A={\Bbb S}^k$ and  $X={\Bbb B}^{k+1}$.  

\begin{theorem} 
\label{th21} 
Let $\mathcal S=\{S_0,\ldots,S_{n+1}\}$ be a cover of \,${\Bbb S}^k$.  
Suppose  the intersection of all the $S_i$ is empty. 
Then $\mathcal S$ can be extended to a cover $\mathcal F$ of ${\Bbb B}^{k+1}$ such that the intersection of all  the $F_i$ is empty if and only if $[\mathcal S]=0$ in $\pi_k({\Bbb S}^n)$. 
\end{theorem}

\begin{proof} 
If $\mathcal S$ can be extended to $\mathcal F$, then we have $\rho_{\mathcal S}:{\Bbb S}^k\to{\Bbb S}^n$ and $\rho_{\mathcal F}:{\Bbb B}^{k+1}\to{\Bbb S}^n$. 
Since $\rho_{\mathcal S}=\rho_{\mathcal F}\circ\imath$  and  $\imath:{\Bbb S}^k\to{\Bbb B}^{k+1}$ is null--homotopic, we have  $[\mathcal S]=[\rho_{\mathcal S}]=0$.  
	
If $[\mathcal S]=0$, then we will show that $\mathcal S$ can be extended to a cover $\mathcal F$.  
From Lemma \ref{prop3} it suffices to prove the theorem for open covers. 
Let $\Phi=\{\varphi_0,\ldots,\varphi_{n+1}\}$ be a partition of unity subordinate to $\mathcal S$. 
Then we have a continuous map 
$$\rho_{\mathcal S,\Phi}:{\Bbb S}^k\to \partial \Delta^{n+1}={\Bbb S}^n,$$ 
where $\rho_{\mathcal S,\Phi}:=\rho_{\mathcal S,\Phi,V}$	(see Section 1) and $V$ is the set of vertices of an $(n+1)$--simplex $\Delta^{n+1}$.

Since $[\rho_{\mathcal S,\Phi}]=0$ in $[{\Bbb S}^k,{\Bbb S}^n]$, there is a homotopy 
$$H:{\Bbb S}^k\times[0,1]\to{\Bbb S}^n,$$ 
where $H(x,0)=\rho_{\mathcal S,\Phi}(x)$, $H(x,1)=v_0$ for all $x$, and $v_0$ is a vertex of $\Delta^{n+1}$. 

Let $L$ be a labeling on $\vrt(\Delta^{n+1})$ such that $L(v_i)=i$. 
Denote
$$U_\ell(\Phi,D):=H^{-1}(U_\ell(\Delta^{n+1}))), \; D:={\Bbb S}^k\times[0,1],$$
where $U_L(\Delta^{n+1}) =\{U_\ell(\Delta^{n+1}),\, \ell=1,\ldots,n+2\}$ (see Section 3). 
It is clear that $\mathcal U_\ell(\Phi,D):=\{U_\ell(\Phi,D)\}$ is an open cover of $D$, $\mathcal U(\Phi,{\Bbb S}^k):=\mathcal U(\Phi,D)|_{{\Bbb S}^k}$ is a cover of ${\Bbb S}^k$, and 
$$U_\ell(\Phi,{\Bbb S}^k)=\{x\in {\Bbb S}^k: \varphi_\ell(x)>0\}\subset S_\ell \; \mbox{ for all } \, \ell.$$

Denote by $\Pi(\mathcal S)$ the set of all partitions of unity subordinate to $\mathcal S$. 
Then for all $\ell$
$$S_\ell=\bigcup\limits_{\Phi\in\Pi(\mathcal S)}{U_\ell(\Phi,{\Bbb S}^k)}. $$

Let 
$$W_\ell=\bigcup\limits_{\Phi\in\Pi(\mathcal S)}{U_\ell(\Phi,D)}. $$
Then $\mathcal W:=\{W_\ell\}$ is an open cover of $D$ that extends $\mathcal S$. 

The boundary of $D$ consists of two components $D_0:={\Bbb S}^k\times\{0\}$ and  $D_1:={\Bbb S}^k\times\{1\}$. 
Actually, $\mathcal W$ on $D_0$ is $\mathcal S$ and $D_1$ is covered only by one set, namely $D_1\subset W_0.$ 
Let $Z$ be a $(k+1)$--disk with boundary $D_1$ and let 
$$B:=D\cup Z, \; \; D\cap Z=D_1. $$
It is clear that $B$ is homeomorphic to ${\Bbb B}^{k+1}$. 
Let $F_0:=W_0\cup Z$ in $B$ and let $\mathcal F:=\{F_0,W_1,\ldots,W_{n+1}\}$, then $\mathcal F$ is a cover of $B$ that extends $\mathcal S$.
\end{proof}

Next consider the case when $A$ is the boundary of a manifold $X$.
\begin{df} 
Let  $\mathcal S=\{S_0,\ldots,S_{n+1}\}$ be a cover of an oriented $n$--dimensional  manifold $N$ without boundary.  
If the intersection of all $S_i$ is empty, then $[\mathcal S]\in{\Bbb Z}=[N,{\Bbb S}^n]$. 
We call $[\mathcal S]$ the degree of $\mathcal S$ and denote it by $\deg(\mathcal S)$. 
\end{df} 

\begin{theorem} 
\label{th22} 
Let $M$ be an oriented $(n+1)$--dimensional manifold with boundary, and let  $\mathcal S=\{S_0,\ldots,S_{n+1}\}$ be a cover of $\partial M$ such that the intersection of all $S_i$ is empty. 
Then $\mathcal S$ can be extended to a cover $\mathcal F$ of $M$, such that all covers $F_i$ have no a common point, if and only if the degree of $\mathcal S$ is zero. 
\end{theorem}

\begin{proof} 
From the {\it Hopf Extension (Degree) Theorem} it follows that a continuous map $f:\partial M\to {\Bbb S}^n$ can be extended to a globally defined continuous map $F: M\to {\Bbb S}^n$, with  $\partial F=f$, if and only if the degree of $f$ is zero. 
It implies that if $\mathcal S$ can be extended, then $\deg(\rho_{\mathcal S})=\deg(\mathcal S)=0.$

If $\deg(\mathcal S)=0$, then the proof that $\mathcal S$ can be extended is almost the same as the proof in Theorem \ref{th21}. 
In the last step we can use the {\it Collar Neighborhood Theorem}  that $\partial M$ has a neighborhood $C$ in $M$ which is homeomorphic to the product $D=\partial M\times[0,1]$. 
Let $F_0:=W_0\cup (M\setminus D)$. 
Then $\mathcal F:=\{F_0,W_1,\ldots,W_{n+1}\}$  is a cover of $M$ that extends $\mathcal S$.
\end{proof}

It is an interesting problem to find extensions of Theorems \ref{th21} and \ref{th22} for general $X, A$ and $V$. 

For extensions of the KKM and Sperner type lemmas we need pairs of spaces $(X,A)$ such that covers of $A$, which are not null--homotopic, cannot be extended to $X$. 
So we need only the ``necessary''parts of Theorems \ref{th21} and \ref{th22}.
Note that pairs of spaces $(X,A)$ in these theorems satisfy the property that any continuous map $f:A\to {\Bbb S}^n$ with $[f]\ne0$ cannot be extended to a continuous $F:M\to {\Bbb S}^n$.

Let $\mathcal S=\{S_0,\ldots,S_{n+1}\}$ be a cover of $A$ and $\mathcal F=\{F_0,\ldots,F_{n+1}\}$ be a cover of $X$.  
Suppose that the intersection of all of the $S_i$ is empty. 
If $\mathcal F$ is an extension of $\mathcal S$ and the intersection of all of the $F_i$ is empty, then we have maps $\rho_{\mathcal S}:A\to{\Bbb S}^n$ and $\rho_{\mathcal F}:X\to{\Bbb S}^n$ such that  $\rho_{\mathcal F}|_A=\rho_{\mathcal S}$. 
This fact motivates the following definition.

\begin{df}
We say that a pair of spaces $(X,A)$, where $A\subset X$, belongs to $\ep_n$ and write $(X,A)\in\ep_n$ if   there is a continuous map $f:A\to {\Bbb S}^n$  with $[f]\ne0$ in $[A,{\Bbb S}^n]$ that cannot be extended to a continuous map $F:X\to {\Bbb S}^n$ with $F|_A=f$. 
\end{df}

We denoted this class of pairs by $\ep$ after S. Eilenberg and L. S. Pontryagin  who initiated obstruction theory in the late 1930s. 
Obstruction theory (see \cite{Hu,Span}) considers homotopy invariants that equals zero if a map can be extended from the $k$--skeleton of $X$ to the $(k+1)$--skeleton and is non-zero otherwise. 

We conclude this section with a theorem that is a simple consequence of obstruction theory.

\begin{theorem} \label{th23}
Let $(X,A)$ be a pair of spaces.
\begin{enumerate}
\item If the embedding $\imath: A\to X$ is null--homotopic and there are non null--homotopic maps $f:A\to{\Bbb S}^n$, then $(X,A)\in\ep_n$. 
In particular, if $\pi_k({\Bbb S}^n)\ne0$, then $({\Bbb B}^{k+1},{\Bbb S}^k)\in\ep_n$.  
\item If $X$ is an oriented $(n+1)$--dimensional manifold and $A=\partial X$, then $(X,A)\in\ep_n$. 
\end{enumerate}
\end{theorem}

\begin{proof} 
1. Assume the converse.  Then $f:A\to {\Bbb S}^n$, with $[f]\ne0$,  can be extended to $F:X\to {\Bbb S}^n$. 
Since $f=F\circ\imath$ and $[\imath]=0$ in $[A,X]$, we have that $[f]=0$ in $[A,{\Bbb S}^n]$ which is a contradiction. 

2. From the Hopf theorem $f:A\to {\Bbb S}^n$ can be extended if and only if $[f]=0$. 
\end{proof}


\section{KKM and Sperner type lemmas}

The KKM (Knaster--Kuratowski--Mazurkiewicz) lemma states: 

\medskip

\noindent{\it If  a simplex $\Delta^m$ is covered by closed sets $C_i$ for $i \in I_m:=\{0,\dots,m\}$ and that for all $J \subset I_m$ the face of $\Delta^m$ that is spanned by the vertices $v_i$ with $i \in J$ is covered by $C_i$ then all the $C_i$ have a common intersection point.}

\medskip

This lemma was published in 1929 \cite{KKM}. 
Actually, the KKM lemma is an extension of Sperner's lemma published one year before in 1928 \cite{Sperner}.

Let $T$ be a triangulation of a simplex $\Delta^m$. 
Suppose that each vertex of $T$ is assigned a unique label from $I_m$. 
A labeling $L$ is called a {\it Sperner labeling} if  the vertices are labelled in such a way that a vertex $u$  of $T$ belonging to a face that is spanned by vertices $v_i$ from $\vrt(\Delta^m)$ for $i \in J \subset I_m$, can only be labelled by $k$ from $J$.  
 Sperner's lemma states:

\medskip

\noindent {\it Every Sperner labeling of a triangulation of $\Delta^m$ contains a cell labelled with a complete set of labels: $\{0,1,\ldots, m\}$.} 

\medskip

We consider extensions of the KKM and Sperner lemmas.

\begin{theorem} 
\label{th31}
Let $(X,A)\in\ep_{m-1}$ and let  $\mathcal S=\{S_0,\ldots,S_{m}\}$ be a cover of $A$ such that the intersection of all $S_i$ is empty and $[\mathcal S]\ne0$ in $[A,{\Bbb S}^{m-1}]$. 
If  $\mathcal F=\{F_0,\ldots,F_{m}\}$ is a cover of $X$ that extends $\mathcal S$, then all the $F_i$ have a common intersection point.
\end{theorem}

\begin{proof} 
Assume the converse. 
Then $\rho_{\mathcal S}:A\to{\Bbb S}^{m-1}$ can be extended to $\rho_{\mathcal F}:X\to{\Bbb S}^{m-1}$ which is a contradiction. 
\end{proof}

Theorem \ref{th23} implies that if $\pi_k({\Bbb S}^{n})\ne0$, then $({\Bbb B}^{k+1},{\Bbb S}^{k})\in\ep_{n}$.

\begin{cor} 
\label{cor41} 
Let $\mathcal F=\{F_0,\ldots,F_{m}\}$ be a cover of \, ${\Bbb B}^{k+1}$ that extends a cover $\mathcal S$ of ${\partial{\Bbb B}^{k+1}}={\Bbb S}^k$. 
If the intersection of all $S_i$ is empty and $[\mathcal S]\ne0$ in $\pi_k({\Bbb S}^{m-1})$, then all the $F_i$ have a common intersection point.
\end{cor} 

Note that for $k=m-1$ we have the KKM lemma. 
Indeed, the assumptions in this lemma yield that $[\mathcal S]=\deg(\mathcal S)=1\in \pi_k({\Bbb S}^{m-1})={\Bbb Z}$. 

It is interesting that this corollary can be non--trivial for $k>m-1$. 
Consider the cover $\mathcal U:=\mathcal U_{L_\tau}(K)=\{U_0(K),U_1(K),U_2(K),U_3(K)\}, \, K={\Bbb S}^3_{12}$, from Example 1.2. 
Since $[\mathcal U]=1\in \pi_3({\Bbb S}^2)={\Bbb Z}$ Corollary \ref{cor41} implies that

\medskip

\noindent{\it If a cover  $\mathcal F=\{F_0,F_1,F_2,F_3\}$ of ${\Bbb B}^4$ is such that $\mathcal F|_{\partial{\Bbb B}^{4}}=\mathcal U$, then the intersection of all $F_i$ is not empty.}

\medskip

However, for $k=m=2$ any cover $\mathcal S=\{S_0,S_1,S_2\}$ of ${\Bbb S}^2$, where the $S_i$ have no a common point, can be extended to a cover $\mathcal F$ of ${\Bbb B}^3$ such that the intersection of all $F_i$ is empty. 
Actually, it follows from the fact that $\pi_2({\Bbb S}^1)=0$.  

Theorems 1.1 and 2.1 imply a condition for the existence of KKM type lemmas for arbitrary positive integers $k$ and $m$. 

\begin{cor} 
For given $k$ and $m$ there is a  cover $\mathcal S=\{S_0,\ldots,S_{m}\}$ of\, ${\Bbb S}^k$ such that the intersection of all $S_i$ is empty and for any cover $\mathcal F$ of ${\Bbb B}^{k+1}$ that extends $\mathcal S$ all the $F_i$ have a common intersection point if and only if $\pi_k({\Bbb S}^{m-1})\ne0$. 
\end{cor}

Now we extend Theorem \ref{th31} for homotopy classes $h(\mathcal S,V,p)$.

\begin{df}
Let $V$ be  a set  of points $v_0,\ldots,v_m$ in  ${\Bbb R}^d$. 
Consider a point  $p\in{\Bbb R}^d$. 
Denote by $\cov_V(p)$ the collection of all subsets $J$ in $I_m$ such that simplices (convex hulls) in ${\Bbb R}^d$ spanned by vertices $\{v_j, \, j\in J\}$ contain $p$.
\end{df} 

It is clear that we have the following
\begin{prop}
Let $\mathcal S=\{S_0,\ldots,S_m\}$ be a cover of a  space $T$. 
Let   $V:=\{v_0,\ldots, v_m\}$ and $p$ be points in ${\Bbb R}^{d}$. 
Then $p\in P_{\mathcal U,V}(T)$ if and only if for any $J\in\cov_V(p)$  the intersection of the $S_i$ whose subindices $i$ are in $J$ is empty.
\end{prop}

\begin{theorem} 
\label{th1}  
Let $(X,A)\in\ep_n$. 
Let  $\mathcal S=\{S_0,\ldots,S_m\}$ and $\mathcal F=\{F_0,\ldots,F_m\}$ be  covers of $A$ and $X$ respectively.  
Let $V:=\{v_0,\ldots, v_m\}$ and $p$ be points in ${\Bbb R}^{n+1}$. Suppose  $\mathcal F$ extends $\mathcal S$,  for all $J\in\cov_V(p)$ the intersection of the $S_j$ whose subindices are in $J$ is empty, and  
$$h(\mathcal S,V,p)\ne0 \mbox{ in } [A,{\Bbb S}^n]. $$ 
Then there is $J\in\cov_V(p)$ such that  
$$\bigcap\limits_{j\in J}{F_j}\ne\emptyset. $$
\end{theorem}

\begin{proof}
Assume the converse. 
Then $p\in{\Bbb R}^{n+1}\setminus\rho_{\mathcal F,V}(X)$. 
Therefore, $f_{\mathcal F,V,p}:X\to {\Bbb S}^n$ is well defined. 
On the other side, it is an extension of the map $f_{\mathcal S,V,p}:A\to {\Bbb S}^n$ with $[f_{\mathcal S,V,p}]\ne0$, which is a contradiction.  	
\end{proof}

Theorem \ref{th1} implies a generalization of Sperner's lemma: 

\begin{theorem} 
\label{thSp}  
Let $X=|K|$ and $A=|Q|$, where $K$ is a simplicial complex and $Q$ is a subcomplex of $K$. 
Suppose $(X,A)\in\ep_n$. Let $L:\vrt(K)\to\{0,1,\ldots,m\}$ be a labeling of $K$.   
Let $V:=\{v_0,\ldots, v_m\}$ and $p$ be points in ${\Bbb R}^{n+1}$. Suppose there are no simplices in $Q$ whose vertices are labeled by $J\in\cov_V(p)$.  
Let 
$$h(Q,L,V,p)\ne0 \mbox{ in } [|Q|,{\Bbb S}^n]. $$ 
Then there is a simplex $s$ in $K$ and there is $J\in\cov_V(p)$ such that vertices of $s$ have labels $J$.  

If $m=n+1$ and $[L]\ne0$ in $[|Q|,{\Bbb S}^n]$, then there is a simplex in $K$ that has all the labels in $0,\ldots,n+1$. 
\end{theorem}

There are many generalizations of the KKM and Sperner lemmas, see \cite{Bacon,DeLPS,KyFan,Meun,Mus,MusSpT,MusQ,MusS,MusVo,Tucker}. 
Part of them follow from Theorems \ref{th1} and \ref{thSp}. 
As example,  we consider here an extension of Tucker--Bacon lemma \cite{Bacon,Tucker} and \cite{MusVo}. 

\begin{cor} 
Let $(X,A)\in\ep_n$. 
Let   $\mathcal F=\{F_1,F_{-1}\ldots,F_n,F_{-n}\}$ be a  cover of $X$ that extends a cover $\mathcal S$ of $A$.
Suppose $S_i\cap S_{-i}=\emptyset$ for all $i$ and $h(\mathcal S,V,O)\ne0$ in $[A,{\Bbb S}^{n-1}]$, where $V:=\{\pm e_1,\ldots,\pm e_n\}$, $e_1,\ldots,e_n$ is an orthonormal basis and O is the origin in ${\Bbb R}^n$. 
Then there is $i$ such that the intersection of $F_i$ and $F_{-i}$ is not empty.
\end{cor}

\begin{proof} 
Note that $\cov_V(O)$ consists of edges that join $e_i$ and $(-e_{i})$. 
Then Theorem \ref{th1} yields a proof. 
\end{proof}


Consider the case when $X=M$ is an oriented manifold of dimension $n+1$ and $A=\partial M$. 
Then $[A,{\Bbb S}^n]={\Bbb Z}$ and for any continuous $f:A\to{\Bbb S}^n$ we have $[f]=\deg{f}$. 
Now we show that it can improve Theorem \ref{th31}. 

Let $s$ be a $d$-simplex. 
We say that $s$ is {\it fully labelled (or colored)} if vertices of $s$ are labeled (colored) by distinct labels $\ell_0,\ldots,\ell_d$.

Let  $T$ be a triangulation of $M$. 
Let  $L:\vrt(M)\to\{0,1,\ldots,n+1\}$ be a labeling of vertices. 
Let $\partial T$ denote the triangulation $T$ on $\partial M$. 
We denote by $\deg(L,\partial T)$ the class $[\partial T,L]$ in $[\partial M,{\Bbb S}^n]$.  

\begin{theorem} 
\label{th34} 
Let $T$ be a triangulation of a manifold $M$ of dimension $n$ with boundary. 
Then for a labeling $L:\vrt(T)\to \{0,1,\ldots,n\}$ the triangulation $T$ must  contain at least $|\deg(L,\partial T)|$ fully colored simplices. 
\end{theorem}

\begin{proof} 
Actually, $L$ implies a piecewise linear map $f_L:T\to\Delta^{n}$, where $f_L=\rho_{\mathcal U_L(T)}$ and $\deg{f_L}=\deg(L,\partial T)$. 
Then any internal point $y$ in $\Delta^{n}$ is regular for $f_L$, the set of preimages $f_{L}^{-1}(y)$ consists of points $u_k\in M$ such that every $u_k$ lies inside of some fully labelled $(n)$-simplex $t_k\in T$, and the sum of the signs of $u_k$ is equal to $\deg{f_L}$. 
It proves the theorem. 
\end{proof}

Let $P$ be a convex polytope in ${\Bbb R}^d$ with vertices $\{v_1,\ldots,v_m\}$. 
Let $T$ be a triangulation of a manifold $M$ of dimension $d$. 
Let  $L:\vrt(T)\to\{1,2,\ldots,m\}$ be a labeling of $T$. 
If, for $u\in \vrt(T)$, we have $L(u)=i$, then we set  $f_{L,P}(u) :=v_i$. 
Therefore, $f_{L,P}$ is defined for all vertices of $T$, and it uniquely defines a simplicial (piecewise linear) map $f_{L,P}:M\to {\Bbb R}^d$.

The following theorem extends Theorems \ref{thSp}, \ref{th34} and the De Loera - Petersen - Su theorem \cite{DeLPS}.

\begin{theorem} 
\label{SpM} 
Let $P$ be a convex polytope in ${\Bbb R}^d$ with $m$ vertices. 
Let  $T$ be a triangulation of an oriented manifold $M$ of dimension $d$ with boundary. 
Let $L:\vrt(T)\to\{1,2,\ldots,m\}$ be a labeling such that $f_{L,P}(\partial M)\subseteq \partial P$. 
Then  $T$ contains  at least $(m-d)|\deg(L,\partial T)|$ fully labelled $d$-simplices.
\end{theorem}

\begin{proof} 
Consider a set of points $S$ in the interior of $P$ so that the interior of every simplex determined by $d+1$ vertices in $V:=\vrt(P)$ contains a unique point from $S$. 
In other words, for any two distinct points $x$ and $y$ in $S$ the intersection of the sets $\cov_V(x)$ and $\cov_V(y)$ is empty.  
Such sets have been called {\it pebble sets} by De Loera, Peterson, and Su. 
In \cite{DeLPS} they proved that in $P$ there is a pebble set of cardinality at least $m-d$. 
Note that $\deg(f_{L,P})=h(\partial T,L,V,p)$ for any internal point $p$ in $P$. 
Let us apply Theorem \ref{thSp} for all points $p$ in $S$. 
Using the same argument about the number of preimages of $f^{-1}_{L,P}(p)$ as in Theorem \ref{th34} we prove the theorem. 
\end{proof}

We conclude this paper by an extension of Theorem \ref{SpM} for simplicial complexes.   
Let  $K$ be a $d$-dimensional simplicial complex. 
E. D. Bloch  \cite{Bloch} defines  the ``boundary'' of $K$, denoted $\bd K$, as the collection of all $(d-1)$-simplices of $K$ that are contained in an odd number of $d$-simplices, together with all the faces of these $(d-1)$-simplices.

Let $P$ be a convex polytope in ${\Bbb R}^d$ with $m$ vertices.  
Any labeling $L:\vrt(K)\to\{1,2,\ldots,m\}$ defines a simplicial map $f_{L,P}:|K|\to P\subset {\Bbb R}^d$. 
So we have a map $f_{L,P}|_{|\bd K|}:\bd(X)\to \partial P\simeq{\Bbb S}^{d-1}$. 
Let us denote the degree of this map modulo 2 by $\dg2(L,\bd K)$. 
From \cite[Theorem 1.5]{Bloch} it follows that the cardinality of $f^{-1}_{L,P}(p)$, where $p$ lies inside of $P$,   is equal to $\dg2(L,\bd K)$  modulo 2. 
Then the Pebble Set Theorem \cite{DeLPS} yields the following theorem.  

\begin{theorem} 
\label{SpB} 
Let $P$ be a convex polytope in ${\Bbb R}^d$ with $m$ vertices. 
Let  $T$ be a triangulation of a simplicial complex $X$ of dimension $d$. 
Let $L:\vrt(T)\to\{1,2,\ldots,m\}$ be a labeling such that $f_{L,P}(|\bd T|)\subseteq \partial P$. If $\dg2(L,\bd T)$ is odd, then $T$ contains  at least $(m-d)$ fully labelled $d$-simplices. 
\end{theorem}

\begin{cor} 
Let $T$ be  a triangulation of a simplicial complex $X$ of dimension $d$. 
If $\dg2(L,\bd T)$ for a labeling $L:\vrt(T)\to \{1,2,\ldots,d+1\}$ is odd, then $T$  must  contain at least one fully colored $d$-simplex. 
\end{cor}

The KKM lemma and its relatives have many applications in several fields of pure and applied mathematics. 
In \cite{MusFD} we consider some extensions of results of this paper that can be applied in game theory and mathematical economics. 

\medskip

\noindent{\bf Acknowledgement.} I  wish to thank Alexander Dranishnikov, James Maissen and Alexey Volovikov for helpful discussions and comments. 

 \medskip

\noindent Oleg R. Musin\\ 
University of Texas Rio Grande Valley, One West University Boulevard, Brownsville, TX, 78520 \\
and\\
IITP RAS, Bolshoy Karetny per. 19, Moscow, 127994, Russia\\
{\it E-mail address:} oleg.musin@utrgv.edu


\begin{thebibliography}{99}
	
\bibitem{Bacon}
P. Bacon, Equivalent formulations of the Borsuk-Ulam theorem, {\it Canad. J. Math.,} {\bf 18} (1966), 492--502.

\bibitem{Bloch}
E. D.~Bloch.
\newblock Mod 2 degree and a generalized no retraction theorem.
\newblock {\em Math. Nachr.}, {\bf 279} (2006), 490--494.

\bibitem{Bryant}
J. L. Bryant, Piecewise linear topology, {\it Handbook of geometric topology,} 219-259, North-Holland, Amsterdam, 2002.

\bibitem{DeLPS}
J. A. De Loera, E. Peterson, and F. E. Su, A Polytopal Generalization of Sperner's Lemma,  {\it J. of Combin. Theory Ser. A,} {\bf 100} (2002), 1-26. 

\bibitem{KyFan}
K. Fan, A generalization of Tucker's combinatorial lemma with topological applications. {\it Ann. of Math.,} {\bf 56} (1952), 431-437.

\bibitem{Hu}
 	S.--T. Hu, Homotopy theory, Academic Press, 1959.

\bibitem{KKM}
B. Knaster, C. Kuratowski, S. Mazurkiewicz, Ein Beweis des Fixpunktsatzes f\"ur $n$--dimensionale Simplexe, {\it Fundamenta Mathematicae}  {\bf 14} (1929): 132--137.

\bibitem{MS}
K. V. Madahar, K. S. Sarkaria, A minimal triangulation of the Hopf map and its 
application, {\it Geom. Dedicata,} {\bf 82} (2000), 105--114.

\bibitem{Meun}
F. Meunier, Sperner labellings: a combinatorial approach, {\it J. of Combin. Theory Ser. A,} {\bf 113} (2006), 1462-1475.

\bibitem{Milnor}
J. W. Milnor, Topology from the differentiable viewpoint, The University Press of Virginia, Charlottesville, Virginia, 1969.

\bibitem{Mus} 
O. R. Musin, Borsuk-Ulam type theorems for manifolds,  {\it Proc. Amer. Math. Soc.}  {\bf 140} (2012), 2551-2560. 

\bibitem{MusSpT} 
O. R. Musin, Extensions of Sperner and Tucker's lemma for manifolds, {\it J. of Combin. Theory Ser. A,} {\bf 132} (2015), 172--187.  

\bibitem{MusQ}
O. R. Musin, Sperner type lemma for quadrangulations, {\it Moscow Journal of Combinatorics and Number theory}, {\bf 5} (2015), 	arXiv:1406.5082. 

\bibitem{MusS}
O. R. Musin, Generalizations of Tucker--Fan--Shashkin lemmas, arXiv:1409.8637, appear in {\it Arnold Math. J.}

\bibitem{MusFD}
O. R. Musin, KKMS type theorems with boundary conditions, in preparation

\bibitem{MusVo}
O. R. Musin and A.\,Yu. Volovikov, Borsuk--Ulam type spaces, arXiv:1507.08872, appear in {\it Moscow Math. J.} 

\bibitem{Span}
E. H. Spanier, Algebraic topology, McGraw-Hill, 1966.

\bibitem{Sperner}
E. Sperner, Neuer Beweis f\" ur die Invarianz der Dimensionszahl und des Gebietes, Abh. Math. Sem. Univ. Hamburg {\bf 6} (1928), 265-272.

\bibitem{Tucker}
A. W. Tucker,  Some topological properties of the disk and sphere. In: Proc. of the First Canadian Math. Congress, Montreal, 285-309, 1945.

 \end{thebibliography}
\end{document}